\documentclass[10pt]{amsart}
\usepackage{caption}
\usepackage{subcaption}
\usepackage{enumerate,amssymb,  mathrsfs, graphicx}
\newtheorem{theorem}{Theorem}[section]
\newtheorem{lemma}[theorem]{Lemma}
\newtheorem*{lemma*}{Lemma}

\newtheorem{proposition}[theorem]{Proposition}

\theoremstyle{definition}

\newtheorem{remark}[theorem]{Remark}
\newtheorem{ques}{Question}[section]

\theoremstyle{remark}

\numberwithin{equation}{section}

\newcommand{\A}{\mathbb{A}}

\newcommand{\X}{\mathbb{X}}

\newcommand{\Y}{\mathbb{Y}}

\newcommand{\onto}{\xrightarrow[]{{}_{\!\!\textnormal{onto\,\,}\!\!}}}

\def\le{\leqslant}
\def\ge{\geqslant}

\def\XXint#1#2#3{{\setbox0=\hbox{$#1{#2#3}{\int}$}\vcenter{\hbox{$#2#3$}}\kern-.5\wd0}}

\def\XXiint#1#2#3{{\setbox0=\hbox{$#1{#2#3}{\iint}$}\vcenter{\hbox{$#2#3$}}\kern-.5\wd0}}

\begin{document}
\title{Radially symmetry of minimizers to the weighted $p-$Dirichlet energy}

\author[D. Kalaj]{David Kalaj}
\address{University of Montenegro, Faculty of natural sciences and mathematics, Podgorica, Cetinjski put b.b. 81000 Podgorica, Montenegro }
\email{davidk@ucg.ac.me}

%    General info
\subjclass[2010]{Primary 35J60; Secondary  30C70}

%\date{\today}

\keywords{Variational integrals, harmonic mappings, energy-minimal deformations, Dirichlet-type energy. }

%\begin{abstract}
%\end{abstract}

\maketitle

%\dedicatory{}
\begin{abstract}
Let $\mathbb{A}=\{z: r< |z|<R\}$ and $\A^\ast=\{z: r^\ast<|z|<R^\ast\}$ be annuli in the complex plane. Let $p\in[1,2]$ and assume that  $\mathcal{H}^{1,p}(\A,\A^*)$  is the class of Sobolev homeomorphisms between $\A$ and $\A^*$,  $h:\A\onto \A^*$. Then we consider the following Dirichlet type energy of $h$:  $$\mathscr{F}_p[h]=\int_{\A(1,r)}\frac{\|Dh\|^p}{|h|^p}, \  \ 1\le p\le 2.$$ We prove that this energy integral attains its minimum, and the minimum is a certain radial diffeomorphism $h:\A\onto \A^*$, provided a radial diffeomorphic minimizer exists. If $p>1$ then such diffeomorphism exists always. If $p=1$, then the conformal modulus of $\A^\ast$ must not be greater or equal to $\pi/2$. This curious phenomenon is opposite to the Nitsche type phenomenon known for the standard Dirichlet energy.
\end{abstract}

%\maketitle \pagestyle{myheadings} \markboth{ David Kalaj}{Dirichlet-type energy of mappings between two concentric annuli}
\section{Introduction} The general law of hyperelasticity tells us that there exists an energy integral
$ E[h] = \int_\X
E(x, h, Dh) dx$  where $E : \mathbb{X} \times \mathbb{Y} \times \mathbb{R}^{n\times n}\to \mathbb{R}$ is a given stored-energy function characterizing
mechanical properties of the material. Here  $\X$ and $\Y$ are nonempty bounded domains in $\mathbb{R}^n
, n > 2.$
The mathematical models of nonlinear elasticity have been first studied by Antman \cite{[2]}, Ball \cite{ball, [5]}, and Ciarlet \cite{[13]}. One of the interesting and important problems in nonlinear
elasticity is whether the radially symmetric minimizers are indeed global
minimizers of the given physically reasonable energy. This leads us to study energy minimal
homeomorphisms $h: \A
\onto \A^\ast$ of Sobolev class $\mathscr{W}^{1,2}$ between annuli
$\A = \A(r, R) = \{x \in\mathbb{R}^n: r < |x| < R\}$ and $\A^\ast
=\A(r_\ast, R_\ast) = \{x \in\mathbb{R}^n: r_\ast < |x| < R_\ast\}$.
Here $0 \le  r < R$ and $0 \le  r_\ast < R_\ast$ are the inner and outer radii of $\A$ and
$\A^\ast$. The variational approach to Geometric Function Theory \cite{atm1,atm}  makes this problem more important. Indeed, several papers are
devoted to understanding the expected radial symmetric properties see \cite{koski2018} and the references therein. Many times experimentally known
answers to practical problems have led us to the deeper study of such mathematically
challenging problems.  We seek to minimize the $p$-harmonic energy of mappings between two annuli in $\mathbb{R}^2$. We consider the modified Dirichlet energy $\mathscr{F}_p[f]=\int_{\A}\frac{\| Df\|^p}{|f|^p}$, $1\le p\le 2$ and minimize it.
\section{$p$-harmonic equation and statement of the main results }\label{sec-1}
For natural number $n$, let
$A=(a_{i,j})_{n\times n}\in \mathbb{R}^{n\times  n}$.
We use $A^{T}$ to denote the transpose of $A$.
The {\it Hilbert-Schmit norm}, also called the {\it Frobenius norm},
of $A$ is denoted by $\|A\|$,
where $$\|A\|^{2}=\sum_{1\leq i,j\leq n} \left| a_{i,j} \right|^2= \text{tr} [A^{T}A].$$

For $p\geq1$, we say that a mapping $h$ belongs to the class $\mathcal{W}^{1, p}(\mathbb{A},\A^\ast)$, if $h$ belongs to the
Sobolev space $\mathcal{W}^{1, p}(\mathbb{A})$ and maps $\mathbb{A}$ onto $\A^\ast$.
Let $h=(h^{1},\ldots,h^{n})$ belong to $\mathcal{W}^{1,p}(\mathbb{A},\A^\ast)$.
We denote the {\it Jacobian matrix} of $h$ at the point $x=(x_{1},\ldots,x_{n})$ by $Dh(x)$, where
$Dh(x)=\left(\frac{\partial h^{i}}{\partial x_{j}} \right)_{n\times n}\in \mathbb{R}^{n\times  n}$. Then
$$
\|Dh\|^{2}=\sum_{1\leq i,j\leq n}
\left|  \frac{\partial h^{i}}{\partial x_{j}}\right|^2.
$$
Here $\frac{\partial h^{i}}{\partial x_{j}}$ denotes the weak partial derivatives of $h^{i}$ with respect to $x_{j}$.
If $h$ is continuous and belongs to
$\mathcal{W}^{1, p}(\mathbb{A},\A^\ast)$ $(p\geq1)$,
then the weak and ordinary partial derivatives coincide a.e. in $\mathbb{A}$ (cf. \cite[Proposition 1.2]{rick}).
Let $h=\rho S$, where $S=\frac{h}{|h|}$ and $\rho=|h|$.
By \cite[Equality (3.2)]{kalaj2018}, we obtain that
$$
Dh(x)=\nabla \rho(x)\otimes S(x) +\rho \cdot DS(x)
$$
and
\begin{equation}\label{eq-1.1}
\|Dh(x)\|^{2}=|\nabla\rho(x)|^{2} +\rho^{2} \|DS(x)\|^{2},
\end{equation}
where $\nabla \rho$ denotes the gradient of $\rho$.

We say that $h:\mathbb{A} \rightarrow\mathbb{A}^{*} $ is a {\it  radial mapping},
 if $h(x)=\rho(|x|)\frac{x}{|x|}$ and if $\rho$ is real and positive function. We use $\mathcal{R}(\mathbb{A},\A^\ast)$  to denote the class of radial homeomorphisms in $\mathcal{W}^{1,p}(\mathbb{A},\A^\ast)$
and use $\mathcal{P}(\mathbb{A},\A^\ast)$ to denote the class of generalized radial homeomorphisms in $\mathcal{W}^{1,p}(\mathbb{A},\A^\ast)$.
We also use $\mathcal{H}(\mathbb{A},\A^\ast)$ to
denote the class of homeomorphisms in $\mathcal{W}^{1,p}(\mathbb{A},\A^\ast)$.

As it is said before, an important problems in nonlinear elasticity is whether the radially symmetric
minimizers are indeed global minimizers.
For example, Iwaniec, and Onninen \cite{iwon} discussed the minimizers of the
following two energy integrals:
$$\mathfrak{E}[h]=\int_{\mathbb{A}} \|Dh(x)\|^{n} dx
\quad\text{and}\quad
 \mathfrak{F}[h]=\int_{\mathbb{A}} \frac{\|Dh(x)\|^{n}}{|h(x)|^{n}} dx
 $$
among all homeomorphisms in $\mathcal{W}^{1,n}(\mathbb{A},\mathbb{A}^{*})$, respectively. The energy integral $\mathfrak{F}$ for $n=2$, has been considered previously by Astala, Iwaniec, and Martin in \cite{atm1}. Further such energy has been generalized in planar annuli by Kalaj in \cite{ka2019,klondon} and spatial annuli in \cite{kalaj2019}.
On the other hand,  Koski and Onninen \cite{koski2018} investigated the minimizers of the $p$-harmonic energy
$$\mathcal{E}_{p}[h]=\int_{\mathbb{A}} \|Dh(x)\|^{p} dx $$
among all homeomorphisms in $\mathcal{W}^{1,p}(\mathbb{A},\mathbb{A}^{*})$,
where $\mathbb{A}$ and $\mathbb{A}^{*}$ are planar annuli and $1\leq p<2$, provided the homeomorphisms fix the outer boundary.
Recently, Kalaj \cite{kalaj2018}  studied the Dirichlet-type energy $\mathscr{F}[h]$ among mappings in $\mathcal{H}(\mathbb{A},\A^\ast)$,
where
\begin{equation}\label{eq-1.2}
 \mathscr{F}[h]=\int_{\mathbb{A}} \frac{\|Dh(x)\|^{n-1}}{|h(x)|^{n-1}}dx.
 \end{equation}
For $n=3$, the author proved that the minimizers of $\mathscr{F}[h]$ are certain generalized radial diffeomorphism (cf. \cite[Theorem 1.1]{kalaj2018}).
Motivated by the case $n=3$, in \cite{kalaj2018} it was posed the following question.

\begin{ques}\label{con-1.1}
For $n\not=3$,
does the Dirichlet integral of $h\in \mathcal{H}(\mathbb{A},\mathbb{A}^{*}) $, i.e. the integral
$$
\mathscr{F}[h]=\int_{\mathbb{A} } \frac{\|Dh(x)\|^{n-1}}{|h(x)|^{n-1}}dx,
$$
achieve its minimum for generalized radial diffeomorphisms between annuli?
\end{ques}
Then in the subsequent paper by Kalaj and Chen \cite{calculus} was given the following answer.
\begin{theorem}\label{thm-1.2}
For $n \geq4$, we have
\begin{eqnarray*}
& &\inf_{h\in \mathcal{H}(\mathbb{A},\mathbb{A}^{*})}
 \mathscr{F}[h]
=\inf_{h\in \mathcal{P}(\mathbb{A},\mathbb{A}^{*})}
 \mathscr{F}[h]
\end{eqnarray*}
The last infimum is never attained.

\end{theorem}
In this paper, we consider the case of the $p-$energy Sobolev $\mathcal{W}^{1,p}$ homeomorphisms between annuli $\A$ and $\A^*$ in the complex plane.
Let $$\mathscr{F}_p[h]=\int_{\A(1,r)}\frac{\|Dh\|^p}{|h|^p}, \  \ 1\le p<2.$$ Then we seek the homeomorphisms $h$ of the class $\mathcal{W}^{1,p}$ which are furthermore assumed to preserve the order of the boundary components $|h(z)|\to $r when $|z|\to r^\ast$ and $|h(z)|\to R^\ast$ when $|z|\to R$. Such a class of Sobolev homeomorphisms with the above property is denoted by $\mathcal{H}^{1,p}(\A,\A^*)$ and we say that they are \emph{admissible homeomorphisms}.
Since we minimize the $\mathscr{F}_p$ energy in the class of homeomorphisms,  we can perform the inner variation of the independent
variable $z_\epsilon = z + \epsilon \tau(z)$, which leads to the system (see for example \cite{kalaj2018})
\begin{equation}\label{div0}\mathrm{div}\left(\frac{1}{|h|^p}\|Dh\|^{p-2}(Dh)^* Dh-\frac{1}{p|h|^p}\|Dh\|^p I\right)=0,\end{equation} where $$\mathrm{div}\left(
                                                                                                                      \begin{array}{cc}
                                                                                                                        a(x,y) & b(x,y) \\
                                                                                                                        c(x,y) & d(x,y) \\
                                                                                                                      \end{array}
                                                                                                                    \right):=\left(
                                                                                                                              \begin{array}{c}
                                                                                                                                a_x+b_y \\
                                                                                                                                c_x+d_y \\
                                                                                                                              \end{array}
                                                                                                                            \right).$$
                                                                                                                            Here $z=(x,y)$.
Our argument does not make direct use of the inner variational equation \eqref{div0}.
Some important facts that follow from \eqref{div0} are as follows.
\begin{enumerate}
\item If we assume that $h$ is radial, then \eqref{div0} reduces to the Euler-Lagrange equation \eqref{afte} below.
\item Further if $f$ is a solution of \eqref{div0} then so is $\tilde f=\frac{1}{f}$.
\item Let $f_1(z)=\frac{1}{r_\ast} f(r z)$. Then $f_1:\A(1,r_1)\onto \A(1,R_1)$, provided that $f:\A(r,R)\onto \A(r^\ast, R^\ast)$, where $R_1=R_\ast/r_\ast$ and $r_1=R/r$. Moreover, $f$ satisfies \eqref{div0} if and only if $f_1$ satisfies the same equation.
    \end{enumerate}
    This is why we reduce the problem to the annuli $\A=\A(1,r)$ and $\A^*=\A(1,R)$.
Now we formulate the main results.
\begin{theorem}\label{th1}  Let $\A$ and $\A^*$ be planar annuli and $1< p\le 2.$ Then  there exists a
radially symmetric mapping $h_\circ:\A\to \A^*$  such that
\begin{equation}\label{slice}
\min_{{\mathcal{H}}^{1,p}(\A,\A^*)} \mathscr{F}_p[h] =\mathscr{F}_p[h_\circ] .\end{equation}
The map $h_\circ$ is the unique minimizer, up to a rotation, in the class $\mathcal{H}^{1,p}(\A,\A^*)$. Furthermore, the
minimizer $h_\circ$ is a homeomorphism.
\end{theorem}

\begin{theorem}\label{th2}  Let $\A$ and $\A^*$ be planar annuli.  Then there exists a
radially symmetric mapping $h_\circ:\A\to \A^*$  which is a homeomorphism such that
\begin{equation}\label{slice1}
\min_{{\mathcal{H}}^{1,1}(\A,\A^*)}\mathscr{F}_1[h] =\mathscr{F}_p[h_\circ],\end{equation}  if and only if \begin{equation}\label{rR}
 {\frac{\pi }{2}-\tan^{-1}\left[\frac{1}{\sqrt{r^2-1}}\right]}\ge \log R.\end{equation}
The map $h_\circ$ is the unique minimizer, up to a rotation, in the class $\mathcal{H}^{1,1}(\A,\A^*)$.
\end{theorem}
\begin{remark} Note that the case $p=2$ of Theorem~\ref{th1} has been already considered by Astala, Iwaniec, and Martin in \cite{atm1}.

 On the other hand side our result can be seen as a variation of minimization property of radial mappings of $p-$Dirichlet energy throughout Sobolev mappings from the unit ball $\mathbb{B}\subset \mathbb{R}^n$ onto the unit sphere $\mathbb{S}^{n-1}$, fixing the boundary. This is an old problem solved by several authors (see for example \cite{bcl}, \cite{jcb}, \cite{hmc}).

Furthermore, as was remarked before,
Koski and Onninen \cite{koski2018} have considered $\mathcal{E}_{p}$ energy and proved the minimization property, under a certain  constrain. Indeed, if we denote the outer boundary of $\A$ by $\partial_\circ \A$ and consider the subfamily of homomorphisms $\mathcal{H}_\circ=\{f\in\mathcal{H}^{1,p}(\A,\A^*): f(x)=\frac{R_\ast}{R} x, \  \text{ for } x\in \partial_\circ \A\}$, then the minimizer of $\mathcal{E}_{p}$ energy is a radial mapping $h(x)=\rho(x)\frac{x}{|x|}$ provided that $R$ and $r$ satisfies some inequality that depends on $p$ (\cite[Theorem~1.5]{koski2018}). In the same paper they proved that this constraint is crucial and there exists annuli, where the minimizer of  $\mathcal{E}_{p}$ is not a   radial mapping.
\end{remark}
\begin{remark}
 By virtue of the density of diffeomorphisms in $\mathcal{H}^{1,p}(\A,\A^\ast)$, see \cite{18,
23}, we can equivalently replace the admissible homeomorphisms by sense preserving
diffeomorphims. Indeed, for $p \ge 1$, we have
\begin{equation}\label{infmin}
\inf_{f\in \mathcal{H}^{1,p}(\A,\A^\ast)}\mathcal{E}_p[h] =
\inf_{f\in \mathrm{Diff}(\A,\A^\ast)}\mathcal{E}_p[h].\end{equation}
Here by $\mathrm{Diff}(\A,\A^\ast)$ we denote the class of orientation preserving diffeomorphisms
from $\A$ onto $\A^\ast$ which also preserve the order of the boundary components.
A similar result hold for the $\mathscr{F}_p$ energy. Indeed \begin{equation}\label{infmin}
\inf_{f\in \mathcal{H}^{1,p}(\A,\A^\ast)}\mathscr{F}_p[h] =
\inf_{f\in \mathrm{Diff}(\A,\A^\ast)}\mathscr{F}_p[h].\end{equation}
\end{remark}

\section{Radial minimizer of the energy $\mathscr{F}_p[h]$, $1<p<2$}
This section aims is to find the radial minimizer $h_\circ$ of $\mathscr{F}_p$ energy that maps annuli $\A(1,r)$ onto $\A(1,R)$ keeping the boundary order. Moreover, we will use that solution to prove the minimization property of $h_\circ$ in the class of all Sobolev homeomorphisms. Contrary to the case $p=1$, which will be considered later, we will not have any restriction on $r$ and $R$.
Assume that $h(z) = H(t) e^{i \theta}$, where $z=t e^{i\theta}$, where $H$ is a differentiable function and that $t\in[1,r]$, $\theta\in[0,2\pi]$. Then $$\|D h\|^2 = |h_t|^2 + \frac{|h_{\theta}|^2}{t^2}=\dot H(t)^2+\frac{H(t)^2}{t^2}.$$
Furthermore $$t\frac{\|Dh\|^p}{|h|^p} = t\left(\frac{1}{t^2}+\frac{ \dot H(t)^2}{H(t)^2}\right)^{p/2}.$$
Let $$L(t, H, \dot H)=t\left(\frac{1}{t^2}+\frac{ \dot H(t)^2}{H(t)^2}\right)^{p/2}.$$
Then Euler-Lagrange equation  $$L_H=\partial_t L_{\dot H},$$ can be written in the following form
\begin{equation}\label{afte}\ddot H= \frac{\dot H \left((p-3) H^3+t H^2 \dot H-t^2 H \dot H^2+(p-1) t^3 \dot H^3\right)}{t H^3+(p-1) t^3 H \dot H^2},\end{equation} where $H=H(t)$, $\dot H= H'(t)$ and $\ddot H=H''(t)$. Then by straightforward calculation \eqref{afte} can be reduced to the following differential equation
\begin{equation}\label{reduced}
\frac{t \dot H(t)}{H(t)}=\frac{\sqrt{g(t)}}{\sqrt{1-g(t)}},\end{equation} where $g$ is a solution to the following differential equation
\begin{equation}\label{gdot}\dot g(t)= F[t, g(t)]:=\frac{2 (2-p) (g(t)-1) g(t)}{t+(p-2) t g(t)}.\end{equation}
Show that $F<0$ provided that $t\ge 1$ and $g(t)\in (0,1)$. Namely $$t+(-2+p) t g(t)\ge t +(p-2)t = (p-1)t>0.$$ Since $2 (2-p) (g(t)-1) g(t)<0$ we infer that $g$ is a decreasing function.

The general solution of \eqref{gdot} is given by $g=k^{-1}$, where the function $k$ is defined by
\begin{equation}\label{inverseg}k(s)=b\exp\left({\frac{(p-1) \log(1-s)-\log s}{2 (2-p)}}\right),\end{equation} where $b$ is a positive constant and $s\in(0,1)$.

By \eqref{reduced} we infer that $H$ is given by \begin{equation}\label{explh}H(t)= C\exp\left[\int_1^t \frac{\sqrt{g(x)}}{\sqrt{1-g(x)} x} \, dx\right].\end{equation}

By using the change $t=k(s)$ in \eqref{explh} we obtain

\begin{equation}\label{change}H(t) = C \exp \left[\int_{g(t)}^{g(1)} \frac{\left(\frac{p-1}{1-s}+\frac{1}{s}\right) \sqrt{s}}{2 (2-p) \sqrt{1-s}} ds\right].\end{equation}
Since we seek increasing homeomorphic mappings $H:[1,r]\onto [1,R]$, we have the initial conditions $H(1)=1$ and $H(r)=R$.  Then $C=1$. Let $0<\tau<1$ and chose $b=b(\tau)$ so that  $$b=\exp\left(\frac{(p-1) \log(1-\tau)-\log \tau}{2 (p-2)}\right).$$ Denote the corresponding $g$ by $g_\tau$. Then we have $g_\tau(1)=\tau$.

Moreover by \eqref{inverseg} $$g_\tau\left[\exp\left({\frac{(p-1) \log(\frac{1-t}{1-\tau})-\log \frac{t}{\tau}}{2 (2-p)}}\right)\right]=t.$$ Define the function

$$\mathcal{R}(\tau)=\exp \left[\int_{g_\tau(r)}^{\tau} \frac{\left(\frac{p-1}{1-s}+\frac{1}{s}\right) \sqrt{s}}{2 (2-p) \sqrt{1-s}} dx\right].$$
Then we also define
$$H_\tau(t) =  \exp \left[\int_{g_\tau(t)}^{\tau} \frac{\left(\frac{p-1}{1-s}+\frac{1}{s}\right) \sqrt{s}}{2 (2-p) \sqrt{1-s}} ds\right].$$ Then $$H_\tau(1)=1$$ and
\begin{equation}\label{rRr}H_\tau (r)=\mathcal{R}(\tau).\end{equation}
Let us show that there is a unique $s_\circ=s(r,\tau)\in (0,\tau)$ such that $B(s_\circ)=0$, where $$B(s):={\frac{(p-1) \log(\frac{1-s}{1-\tau})-\log \frac{s}{\tau}}{2 (2-p)}}-\log r.$$ Note that $B$ is continuous, $B(\tau)=0$ and $B(0)=+\infty$. Moreover $$B'(s)=\frac{1+(-2+p) s}{2 (2-p) (-1+s) s}<0.$$ Thus there is a unique $s_\circ$ so that $B(s_\circ)=0$. Then $g_\tau(r)=s_\circ$. Since for $0<s<\tau$ and $p\in(1,2]$, we have $${\frac{ \log(\frac{1-s}{1-\tau})-\log \frac{s}{\tau}}{2 (2-p)}}\ge {\frac{(p-1) \log(\frac{1-s}{1-\tau})-\log \frac{s}{\tau}}{2 (2-p)}},$$ it follows that

$${\frac{ \log(\frac{1-s_\circ}{1-\tau})-\log \frac{s_\circ}{\tau}}{2 (2-p)}}-\log r\ge B(s_\circ)={\frac{(p-1) \log(\frac{1-s_\circ}{1-\tau})-\log \frac{s_\circ}{\tau}}{2 (2-p)}}-\log r=0.$$
Thus
\begin{equation}\label{estimatetc}0<s_\circ<\tau_\circ=\frac{1}{1+r^{4-2 p} \left(-1+\frac{1}{\tau}\right)}.\end{equation}
Then $$\mathcal{R}(\tau)=\exp \left[\int_{s_\circ}^{\tau} \frac{\left(\frac{p-1}{1-s}+\frac{1}{s}\right) \sqrt{s}}{2 (2-p) \sqrt{1-s}} ds\right].$$

Let us show now that, if $p>1$, then for every $R\in (1,+\infty)$, there is $\tau \in (0,1)$ so that $\mathcal{R}(\tau)=R$. It is clear that $\mathcal{R}$ is continuous and also it is clear that $\lim_{\tau \to 0} \mathcal{R}(\tau)=1$.  Let us show that $\lim_{\tau \to 1} \mathcal{R}(\tau)=+\infty$.  Observe that $0\le s\le \sqrt{s}\le 1$. Then from \eqref{estimatetc} we have that $$\mathcal{R}(\tau)\ge K(\tau),$$ where
$$K(\tau)=\exp \left[\int_{\tau_\circ}^{\tau} \frac{\left(\frac{p-1}{s-1}+\frac{1}{s}\right) s}{2 (2-p) \sqrt{1-s}} ds\right].$$
Then $K(\tau)=\exp(k(\tau)-k(\tau_0))$, where $$k(s)=\frac{3+p (s-2)-2 s}{(2-p) \sqrt{1-s}}.$$
Then $$\lim_{\tau \to 1^-}\sqrt{1-\tau} \log K(\tau)=\frac{(p-1) \left(r^2+r^p\right)}{(2-p) r^2}.$$ We notice that here is the moment where $p\in(1,2)$ is an important assumption.
In particular $\lim_{\tau\to 1} R(\tau)=\infty$. So there is $\tau=\tau(r, R)$ so that $\mathcal{R}(\tau)=R$. In view of \eqref{rRr}, we have constructed a smooth increasing mapping $H_\circ=H_{r,R}:[1,r]\to [1,R]$ so that $H(1)=1$ and $H(r)=R$. Let us show that
\begin{equation}\label{hcirc} h_\circ(z)= H(t)e^{i\theta}, \ \  z=t e^{i\theta},\end{equation} is the minimizer in the class of radial homeomorphisms between $\A$ and $\A^\ast$.

Assume now that $H: [1,r]\to [1,R]$ is any smooth homeomorphism and assume that $h(z)=H(t)e^{i\theta}$. Prove that
\begin{equation}\label{prove}\mathscr{F}_p[h]\ge \mathscr{F}_p[h_\circ].\end{equation}
We start from a simple inequality from \cite{koski2018} \begin{equation}\label{koski}
(a+b)^{q/2}\ge s^{1-q/2}a^{q/2}+(1-s)^{1-q/2}b^{q/2}, \ \ q\in[1,2],  \ \ s\in[0,1].
\end{equation}
By inserting $q=p$, $s=g(t)$, $$a=t^{\frac{2}{p}-2}, \ \  b=t^{2/p} \frac{\dot H^2}{H^2}$$ in \eqref{koski} we have
\begin{equation}\label{ejte}\begin{split}t \left(\frac{1}{t^2}+ \frac{\dot H^2}{H^2}\right)^{p/2}&= \left(t^{2/p-2}+ t^{2/p}\frac{\dot H^2}{H^2}\right)^{p/2}\\&\ge (1-g(t))^{1-p/2}t^{1-p}+g(t)^{1-p/2}t\frac{|\dot H|^p}{|H|^p}.\end{split}\end{equation}

The equality in \eqref{koski} is attained precisely when $$\frac{b}{a}= \frac{s}{1-s}$$ and thus  the equality is attained in \eqref{ejte} precisely when \begin{equation}\label{ocndi}\frac{t\dot H}{H}=\frac{\sqrt{g(t)}}{\sqrt{1-g(t)}}.\end{equation}
Then by \begin{equation}\label{preci}a^p\ge  p x^{p-1} a -(p-1) x^p,\end{equation}
where $a=\frac{\dot H(s)}{H(s)}$ and $x=\frac{\sqrt{g(s)}}{\sqrt{1-g(s)} s}$
we get
\begin{equation}\label{seconda} t\left(\frac{1}{t^2}+ \frac{\dot H(t)^2}{H(t)^2}\right)^{p/2}\ge t^{1-p}\frac{1-pg(t)}{(1-g(t))^{p/2}}+\frac{s^{2-p}\sqrt{g(t)}}{(1-g(t))^{(p-1)/2}}p \frac{\dot H}{H}.\end{equation}
Notice that, the condition \eqref{ocndi} is precisely satisfied when we have equality in \eqref{seconda}.

Define $$P(t) = t^{2-p} \left({1-g(t)}\right)^{\frac{1}{2} (1-p)} \sqrt{g(t)},$$ and show that it is a constant. \emph{This fact is crucial for our approach.}

By \eqref{gdot} we obtain that $$\frac{P'(t)}{P(t)}=\frac{2-p}{t}+\frac{(1+(p-2) g(t)) \dot g(t)}{2 (1-g(t)) g(t)}=0.$$ Thus
\begin{equation}\label{PP}
P(t)\equiv c=P(r)=r^{2-p} \left({1-g(r)}\right)^{\frac{1}{2} (1-p)} \sqrt{g(r)}.\end{equation}

Observe that $$g(r)=g_\tau(r)=c_\circ(r,\tau)=c_\circ(r, \tau(r,R)).$$ Thus $c=c(r,R)$.
Now we have \[\begin{split}\mathscr{F}_p[h] &=2\pi \int_1^rt\left(\frac{1}{t^2}+\frac{\dot H(t)^2}{H^2(t)}\right)^{p/2}dt \\&\ge 2\pi \int_1^r\left(t^{1-p}\frac{1-pg(t)}{(1-g(t))^{p/2}}+c(r,R) \frac{\dot H}{H}\right)dt
\\&=2\pi \int_1^r\left(t^{1-p}\frac{1-pg(t)}{(1-g(t))^{p/2}}\right) dt+2\pi\int_1^r c(r,R)\frac{\dot H(t)}{H(t)}dt
\\&=2\pi \int_1^r\left(t^{1-p}\frac{1-pg(t)}{(1-g(t))^{p/2}}\right) dt+2\pi c(r, R)\log R\\&=\mathscr{F}_p[h^\circ].
\end{split}
\]

\section{Radial minimizers for the case $p=1$}\label{sect}

The corresponding subintegral expression for the functional $\mathscr{F}_1[h]=\int_{\A(1,r)} \frac{|Df(z)|}{|f(z)|}$, for radial function $h(z)=H(t)e^{i\theta}$, $z=te^{i\theta}$ is given by  $$L(t, H, \dot H)=\left(1+\frac{t^2 \dot H(t)^2}{H(t)^2}\right)^{1/2}.$$ The corresponding differential equation \eqref{afte} for $p=1$ reduces to
\begin{equation}\label{incsol}\left(-t H(t) \dot H(t)^2+t^2 \dot H(t)^3+H(t)^2 \left(2 \dot H(t)+t \ddot H(t)\right)\right)=0\end{equation} which can be written in the following form  $$\frac{t\dot H(t)}{H(t)}  =\frac{\sqrt{g(t)}}{\sqrt{1-g(t)}}$$  where $g$ is a solution of the differential equation (see \eqref{gdot} for $p=1$):
\begin{equation}\label{eqgen}2 g(t)+t \dot g(t)=0.\end{equation} Then the general solution of \eqref{eqgen} is given by $g(t) = {b t^{-2}}.$ Then the solution of \eqref{incsol} is the solution of the equation $$\frac{t\dot H(t)}{H(t)}=\frac{1}{\sqrt{b^2 t^2-1}}$$ and it is given by $$H(t)=c \exp\left({-\cot^{-1}\left[\sqrt{b^2t^2-1 }\right]}\right).$$
If we let that $H(1)=1$ then
\begin{equation}\label{generalH}H(t)=\exp\left(\cot^{-1}\left[\sqrt{b^2 -1}\right]-\cot^{-1}\left[\sqrt{ b^2 t^2-1}\right]\right).\end{equation}
Here $b\ge 1$. Moreover, if we assume that $H(r)=R$, then after straightforward computations we get
$$b=\frac{\sqrt{\left(1+r^2-2 r \cos\log R\right)} \csc[\log R]}{r}.$$ The corresponding minimizer is denoted by $h_\circ(z)=H(r)e^{i\theta},$ $z=re^{i\theta}$.
Hence  $$\mathscr{F}[h]=2\pi \int_1 ^r \left(1+\frac{t^2(\dot H(t))^2}{H^2(t)}\right)^{1/2}dt \ge 2\pi\int_1^r \sqrt{1-\frac{1}{b^2 t^2}}+\frac{ \dot H(t)}{bH(t)}dt$$

Thus $$\mathscr{F}[h]\ge \mathscr{F}[h_\circ]$$ where $$\mathscr{F}[h_\circ]=2\pi \frac{-\sqrt{b^2-1}+\sqrt{b^2r^2-1}-\csc^{-1}\left[b\right]+\csc^{-1}\left[b r\right]}{b}+ \frac{2\pi \log R}{b}.$$
\begin{lemma}\label{mela}
It exists a radial homeomorphism $h: \A(1,r)\to \A(1,R)$ if and only if $${\frac{\pi }{2}-\tan^{-1}\left[\frac{1}{\sqrt{r^2-1}}\right]}>\log R.$$
\end{lemma}

\begin{proof}
By differentiating \eqref{generalH} w.r.t. $b$ we get
$$\partial_b H(t)=\frac{\exp\left({\cot^{-1}\left[\sqrt{-1+b^2}\right]-\cot^{-1}\left[\sqrt{-1+b^2 t^2}\right]}\right) \left(-\frac{1}{\sqrt{-1+b^2}}+\frac{1}{\sqrt{-1+b^2 t^2}}\right)}{b}.$$ Hence $H$ is decreasing in $b$. The largest value is for $b=1$ and it is equal to

$$R_\circ(r):=\exp\left(\frac{\pi }{2}-\tan^{-1}\left[\frac{1}{\sqrt{r^2-1}}\right]\right)$$ for $t=r$. In other words, there is a increasing diffeomorphism of $[1,r]$ onto $[1,R]$ if and only if $R\le R_\circ(r).$
\end{proof}
\begin{remark}
Observe that $\lim_{r\to \infty} \mathcal{R}(r)=e^{\pi/2}$, so there is not any homeomorphic minimizer of the $\mathscr{F}$ between annuli $\A(1,r)$ and $\A(1,e^{\pi/2})$. Note that the conformal modulus of $\mod{\A(1,e^{\pi/2})}$ is $\log e^{\pi/2}=\pi/2$.  So the case $p=1$ differs from the case $p>1$. Moreover, this case is also opposite to the Nitsche type phenomenon for Dirichlet energy $\mathcal{E}$. Namely Nitsche type phenomenon asserts that $R$ could be arbitrarily large, but not small enough.
\end{remark}

\begin{figure}
\centering
\includegraphics{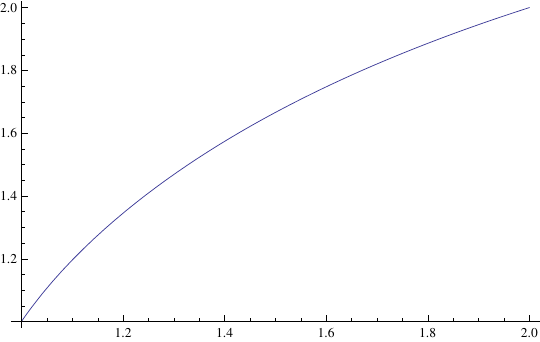}
\caption{The graphic of $H_\circ$ satisfying initial conditions $H(1)=1$, $H(2)=2$ is far from being identity. }
\label{21}
\end{figure}

\section{Proof of Theorem~\ref{th1} and Theorem~\ref{th2}}
We begin with the following proposition
\begin{proposition}\label{provo}

Assume that $h=\rho(z)e^{i\Theta(z)}$ is a diffeomorphism between annuli $\mathbb{A}(1,r)$ and $\mathbb{A}(1,R)$. Then for every $t\in[1,r]$ and $\theta\in [0,2\pi]$ we have
\begin{equation}\label{Theta}
\int_{t\mathbb{T}}|\nabla \Theta(z)||dz|\ge 2\pi.
\end{equation}
If the equality hold in \eqref{Theta} for every $\theta\in[0,2\pi]$, then $\Theta(z) =e^{i\varphi(\theta)},$ $z=te^{i\theta}$, for a diffeomorphism $\varphi:[0,2\pi]\onto[\alpha, 2\pi+\alpha]$.
Further, we have
\begin{equation}\label{Rho}
\int_{1}^{R}\frac{|\nabla \rho(te^{i\theta}) |}{\rho(te^{i\theta})}dt\ge \log R.
\end{equation}
If the equality hold in \eqref{Rho} for every $t\in[1,R]$, then $\rho(te^{i\theta}) =\rho(t)$.

\end{proposition}
\begin{proof}[Proof of Proposition~\ref{provo}]
 First of all, for fixed $t$,  $\gamma(\theta)=e^{i\Theta(t e^{i\theta})}$ is a surjection of $[0,2\pi]$ onto $\mathbb{T}=\{z:|z|=1\}$. Further $$|\nabla \Theta(t e^{i\theta})|^2=|\Theta_t|^2+\frac{|\Theta_\theta|^2}{t^2}. $$

So \begin{equation}\label{gthe}|\gamma'(\theta)|=|\Theta_\theta|\le t |\nabla \Theta(t e^{i\theta})|.\end{equation} The equality is attained in \eqref{gthe} if and only if $\Theta_t\equiv 0$. In this case $\gamma(\theta)=e^{i\varphi(\theta)}$, for a smooth function of $\varphi:[0,2\pi] \onto [\alpha,2\pi+\alpha]$.

We obtain that $$|\mathbb{T}|=2\pi\le  \int_0^{2\pi}|\gamma'(\theta)|d\theta\le \int_{t\mathbb{T}}|\nabla \Theta(z)| |dz|,$$ with an equality if and only if $\Theta(se^{i\theta})$ does not  depend on $t$.  Thus the first statement of the proposition is proved.

Similarly the function $\alpha(t)=\log\rho(t e^{i\theta})$ is a surjection of $[1,r]$ onto $[0,\log R]$ and hence $$\log R= \int_1^r \alpha'(t)dt\le \int_1^r \frac{|\nabla \rho(te^{i\theta})|}{\rho(te^{i\theta})}dt.$$ The equality statement can be proved in the same way as the former part. We only need to use the formula
$$|\nabla \rho(t e^{i\theta})|^2=|\rho_t|^2+\frac{|\rho_\theta|^2}{t^2}\ge |\rho_t|^2. $$
\end{proof}
\begin{proof}[Proof of Theorem~\ref{th1}]
Assume as before that $h(z)=\rho(z)e^{i\Theta(z)}$ is a mapping from the annulus  $\mathbb{A}$ onto the annulus $\mathbb{A}^{*}$.
We start from the following inequality which follows from H\"older inequality

$$\mathscr{F}_p[h]=\int_{\A(1,r)}\frac{\|Dh\|^p}{|h|^p}\ge \frac{\left(\int_{\A(1,r)}\frac{\|Dh\|}{|h|} \cdot \frac{\|Dh_\circ\|^{p-1}}{|h_\circ|^{p-1} }\right)^p}{\left(\int_{\A(1,r)}\frac{\|Dh_\circ\|^p}{|h_\circ|^p}\right)^{p-1}} .$$

In view of \eqref{eq-1.1}
$$
\|Dh\|^{2}=|\nabla \rho|^2+\rho^2|\nabla \Theta|^2,
$$
where $\rho(z)=|h(z)|$.
And thus
$$\frac{\|Dh\|}{|h|}=\left({|\nabla \Theta|^2}+\frac{|\nabla \rho |^2}{\rho^2}\right)^{1/2}.$$ Then by \eqref{prove}, for $q=1$ we have

\begin{equation}\label{oni}
\frac{\|Dh\|}{|h|}\ge \left(\sqrt{1-g(t)}|\nabla \Theta|+\sqrt{g(t)}\frac{|\nabla \rho|}{\rho}\right).
\end{equation}
From \eqref{oni} we get
\[\begin{split}\int_{\A(1,r)}\frac{\|Dh\|}{|h|} &\cdot \frac{\|Dh_\circ\|^{p-1}}{|h_\circ|^{p-1} }
\\&=  \int_{\A(1,r)} \left({|\nabla \Theta|^2}+\frac{|\nabla \rho |^2}{\rho^2}\right)^{1/2} \cdot \frac{\|Dh_\circ\|^{p-1}}{|h_\circ|^{p-1} }
\\&\ge \int_0^{2\pi}\int_1^r  t\frac{\|Dh_\circ\|^{p-1}}{|h_\circ|^{p-1} }\left(\sqrt{1-g(t)}|\nabla \Theta|+\sqrt{g(t)}\frac{|\nabla \rho|}{\rho}\right)dtd\theta.
\end{split}\]
Let $$K(t) = t\sqrt{g(t)}\frac{\|Dh_\circ\|^{p-1}}{|h_\circ|^{p-1}}.$$ Then $$K(t)=t \sqrt{g(t)} \left(\frac{1}{t^2}+\frac{g(t)}{t^2 (1-g(t))}\right)^{\frac{1}{2} (p-1)}=P(t).$$ Thus we again use \eqref{PP} to conclude that $K(t)=c(r,R)$. Furthermore
\[\begin{split}t\frac{\|Dh\|}{|h|} \cdot \frac{\|Dh_\circ\|^{p-1}}{|h_\circ|^{p-1}}&\ge t  \left({t^2 (1-g(t))}\right)^{\frac{1}{2} (1-p)}\left[\sqrt{1-g(t)}|\nabla \theta|+\sqrt{g(t)}\frac{|\nabla \rho|}{\rho}\right]
\\&=t^{2-p}   \left(1-g(t))\right)^{1-p/2}|\nabla \Theta|+c(r,R) \frac{|\nabla \rho|}{\rho}.\end{split}\]
Now by Proposition~\ref{provo} we have
$$\int_{\A}  \frac{|\nabla \rho|}{|\rho|}\ge 2\pi \log R$$
and $$t\int_0^{2\pi}|\nabla \Theta(t e^{i\theta})|d\theta\ge 2\pi.$$
So we have
$$\int_{\A(1,r)}\frac{\|Dh\|}{|h|} \cdot \frac{\|Dh_\circ\|^{p-1}}{|h_\circ|^{p-1} }\ge 2\pi \left(c \log R +\int_1^r t^{2-p}   (1-g(t))^{1-p/2}dt\right)=\mathscr{F}_p[h_\circ].
$$
Thus $$\mathscr{F}_p[h]\ge \frac{\mathscr{F}^p_p[h_\circ]}{\mathscr{F}_p^{p-1}[h_\circ]}=\mathscr{F}_p[h_\circ].$$
The uniqueness part of this theorem follows from Proposition~\ref{provo}. The equation in \eqref{oni} is satisfied if and only if $$\frac{\rho(te^{i\theta})|\nabla \Theta(te^{i\theta})|}{|\nabla \rho(te^{i\theta})|}$$ is a function that depends only on $t$. Since $\Theta(\theta)=e^{i\varphi(\theta)}$, we get $|\nabla \Theta(\theta)|=\varphi'(\theta)=\mathrm{const}$. Because $\varphi:[0,2\pi]\onto [\alpha, 2\pi+\alpha]$, it follows that $\varphi(\theta)=\theta+\alpha$. In other words $h(z)$ is a minimizer if and only if $h(z) = H_\circ(t)e^{i(\theta+\alpha)}=e^{i\alpha} h_\circ(z)$. This finishes the proof.
\end{proof}
\begin{proof}[Proof of Theorem~\ref{th2}]
The proof  of Theorem~\ref{th2} is the same as the proof of Theorem~\ref{th1} up to the part concerning the existence of the radial solutions given in Section~\ref{sect} (See Lemma~\ref{mela}).
\end{proof}

\end{document}